\documentclass{amsart}

\usepackage{amsmath,amsfonts,amssymb,mathabx,color,extdash,mathtools} 
\newcommand{\e}{\varepsilon}
\newcommand{\iy}{\infty}
\newcommand{\st}{\ : \ }

\DeclareMathOperator{\rank}{rank}

\usepackage{amsmath}
\usepackage{amssymb}  
\usepackage{amsthm}
\usepackage{amsfonts}
\usepackage{graphicx}
\usepackage{cancel}
\usepackage{bbm}
\usepackage{url}
\usepackage{enumerate} 
\usepackage{ upgreek }
\usepackage{dsfont}
\usepackage{color}
\usepackage{subcaption}

\usepackage{makecell}
\usepackage{diagbox}
\usepackage{wrapfig}
\usepackage{rotating}
\usepackage{tabularx}
\usepackage{mathrsfs}
\usepackage{tikz}
\usepackage{thmtools}

\usepackage{hyperref}
\hypersetup{
    colorlinks=true,
    linkcolor=blue,
    citecolor=red,
    filecolor=magenta,      
    urlcolor=cyan,
}

\renewcommand{\leq}{\leqslant}
\renewcommand{\geq}{\geqslant}

\newcommand{\R}{\mathbf{R}}
\newcommand{\C}{\mathbf{C}}
\newcommand{\CC}{\mathsf{C}}
\renewcommand{\L}{\mathsf{L}}

\newcommand{\Oct}{\mathbf{O}}
\newcommand{\quat}{\mathbf{H}}

\newcommand{\N}{\mathbb{N}}

\DeclareMathOperator{\conv}{\mathrm{conv}}
\DeclareMathOperator{\range}{\mathrm{range}}

\DeclareMathOperator{\mathspan}{\mathrm{span}}

\DeclareMathOperator{\PSD}{\mathsf{PSD}}
\DeclareMathOperator{\Her}{\mathsf{Her}}

\DeclareMathOperator{\inter}{\mathrm{int}}
\DeclareMathOperator{\relint}{\mathrm{relint}}

\DeclareMathOperator{\face}{\mathrm{face}}

\renewcommand{\int}{\mathrm{int}}

\DeclareMathOperator{\cone}{\mathrm{cone}}

\theoremstyle{plain}
\newtheorem{theorem}{Theorem}
\newtheorem*{theorem*}{Theorem}

\newtheorem{proposition}{Proposition}
\newtheorem*{proposition*}{Proposition}

\newtheorem*{claim*}{Claim}
\newtheorem*{lemma*}{Lemma}

\newtheorem*{conjecture*}{Conjecture}

\theoremstyle{remark}

\newtheorem{example}{Example}

\DeclareMathOperator{\SCF}{\mathscr{S}}

\def\beq{\begin{equation}}
\def\eeq{\end{equation}}
\def\bq{\begin{quote}}
\def\eq{\end{quote}}
\def\ben{\begin{enumerate}}
\def\een{\end{enumerate}}
\def\bit{\begin{itemize}}
\def\eit{\end{itemize}}

\def\ra{\rightarrow}

\def\lb{\left(}
\def\rb{\right)}
\def\lset{\lbrace}
\def\rset{\rbrace}

\def\r|{\right|}

\def\ident{\textnormal{id}}

\DeclareMathOperator{\Pos}{Pos}

\DeclareMathOperator{\EB}{EB}
\DeclareMathOperator{\LorFact}{LorFact}

\DeclareMathOperator{\LorEB}{LorEB}

\DeclareMathOperator{\EA}{EA}

\theoremstyle{plain}
\newtheorem{thm}{Theorem}
\newtheorem{lem}[thm]{Lemma}
\newtheorem{cor}[thm]{Corollary}
\newtheorem{prop}[thm]{Proposition}

\theoremstyle{definition}

\newenvironment{theoremrestated}[1]{%
  \begingroup
  \renewcommand{\thetheorem}{\ref{#1}}%
  \begin{theorem}
}{%
  \end{theorem}
  \endgroup
}

\usepackage[dvipsnames]{xcolor}
\newcommand{\sqmap}[4]{%
\left\langle 
\begin{smallmatrix} 
 #1 &  #2 \\
 #3 &  #4
\end{smallmatrix}
\right\rangle
}

\newcommand{\Sqmap}[4]{
\left\langle 
\begin{matrix}
    #1 & #2 \\ #3 & #4
  \end{matrix}\right\rangle 
}

\author{Guillaume Aubrun}
\address{\small{Institut Camille Jordan, Universit\'{e} Claude Bernard Lyon 1, 43 boulevard du 11 novembre 1918, 69622 Villeurbanne cedex, France}}
\email{aubrun@math.univ-lyon1.fr}
\author{Francesca La Piana}
\address{\small{Department of Mathematics, University of Oslo, P.O. box 1053, Blindern, 0316 Oslo, Norway}}
\address{\small{Department of Mathematics, Saarland University, Campus Saarbr\"ucken, 66123 Saarland, Germany}}

\email{francesca.la-piana@math.uni-sb.de}
\author{Alexander M\"uller-Hermes}
\address{\small{Department of Mathematics, University of Oslo, P.O. box 1053, Blindern, 0316 Oslo, Norway}}
\email{muellerh@math.uio.no, muellerh@posteo.net}

\title{Factorization through Lorentz cones}

\begin{document}

\begin{abstract}
A pair of proper cones $(\CC_1,\CC_2)$ is said to have the Lorentz factorization property (LFP) if every $(\CC_1,\CC_2)$-positive map factors through a direct sum of Lorentzian cones, i.e., cones over Euclidean balls. Clearly, $(\CC_1,\CC_2)$ has the LFP if either $\CC_1$ or $\CC_2$ is a direct sum of Lorentzian cones, and our main goal is to find other examples. We show that such examples cannot be found for pairs $(\CC_1,\CC_2)$ where $\CC_1=\CC_2$, or in the case where both $\CC_1$ and~$\CC_2$ are polyhedral. We also focus on the case where $\CC_1=\CC_\square$ is the square-based cone in $\R^3$. Here, we show that $(\CC_\square,\CC)$ has the LFP whenever~$\CC$ is a symmetric cone, i.e., a direct sum of Lorentz cones, cones of positive semidefinite matrices over the real numbers, complex numbers or quaternions, and the cone of $3\times 3$ positive semidefinite matrices over the octonions. We leave open the question whether there are more examples, but we show that this list cannot be extended by any strictly convex cone $\CC$ or for a cone $\CC$ with $\dim(\CC)\leq 5$. Finally, we discuss an application to a problem in quantum information theory.
\end{abstract}

\maketitle

\tableofcontents

\section{Introduction}

Given two finite-dimensional real vector spaces $V_1$, $V_2$ and proper cones $\CC_1 \subseteq V_1$ and $\CC_2 \subseteq V_2$, we denote by $\Pos(\CC_1,\CC_2)$ the set of $(\CC_1,\CC_2)$-positive maps, i.e., linear maps $P : V_1 \to V_2$ such that $P(\CC_1) \subseteq \CC_2$. We will simply call a map positive if it is clear from context which cones are involved. The set $\Pos(\CC_1,\CC_2)$ is a proper cone in the vector space $\mathcal{L}(V_1,V_2)$. An important example of a proper cone is given by the $n$-dimensional \emph{Lorentz cone} defined for any $n\in \N^*$ as
\[ \L_{n} = \{ (t,x_1,\dots, x_{n-1}) \in \R^n  \st t \geq (x^2_1+\cdots +x^2_{n-1})^{1/2} \}.\]
We say that a cone is \emph{Lorentzian} if it is isomorphic to $\L_n$ for some $n\in \N$. Equivalently, a proper cone $\CC \subset V$ is Lorentzian if and only if there exists on $V$ a quadratic form $q$ of signature $(1,\dim(V)-1)$ such that $\CC \cup (-\CC) = \{x \in V \st  q(x) \geq 0\}$.
In~\cite{aubrun2023annihilating,LPMH25} we introduced the cone $\LorFact(\CC_1,\CC_2)$ of \emph{Lorentz factorizable maps} as
\[ \LorFact(\CC_1,\CC_2) = \conv \{\beta \alpha \st \alpha \in \Pos(\CC_1,\L), \ \beta \in \Pos(\L,\CC_2), \ \L \textnormal{ Lorentzian cone} \} .\]
The cone $\LorFact(\CC_1,\CC_2)$ is also proper in the vector space $\mathcal{L}(V_1,V_2)$ and in particular it is closed (see~\cite{LPMH25}). Clearly, we have $\LorFact(\CC_1,\CC_2)\subseteq \Pos(\CC_1,\CC_2)$ and it is a natural question when the containment is strict. We say that a pair of proper cones $(\CC_1,\CC_2)$ has the \emph{Lorentz factorization property} (LFP) if $\LorFact(\CC_1,\CC_2)= \Pos(\CC_1,\CC_2)$, i.e., if every $(\CC_1,\CC_2)$-positive map factors through Lorentz cones.

Similar concepts have been studied in the context of Banach spaces. The Hilbert space factorization norm $\gamma_2$, which goes back to Grothendieck, quantifies the factorization of linear operators between Banach spaces through Hilbert spaces (see~\cite[\textsection 13]{tomczak1989banach}). The notion of Lorentz factorizable maps is a generalization of the concept of Hilbert-space factorizable operators to the setting of ordered vector spaces (see also~\cite{LPMH25} for a quantitative connection for cones over finite-dimensional normed spaces). The analogue of the LFP in the setting of the isometric theory of Banach spaces has been studied in~\cite{AMH25}.

Our main motivation for studying when pairs of cones have the LFP comes from a problem inspired by quantum information theory. In~\cite{aubrun2023annihilating} two of the authors introduced the notion of \emph{resilience} for a pair of cones $(\CC_1,\CC_2)$, where a certain analogue of a well-known problem in quantum information theory has a negative answer. They also showed that a pair $(\CC_1,\CC_2)$ having the LFP implies that the pair $(\CC_2,\CC_1)$ is resilient. For more information on resilience see Section \ref{section:resilence}. 

\subsection*{Main results}

Clearly, the pair $(\CC_1,\CC_2)$ has the LFP if $\CC_1$ or $\CC_2$ is Lorentzian and more generally if $\CC_1$ or $\CC_2$ is isomorphic to a direct sum of Lorentz cones. It is expected that having the LFP is quite restrictive. Our first main result confirms this intuition:

\begin{thm}\label{thm:SameCone}
Let $\CC$ be a proper cone. The following are equivalent
\begin{enumerate}
\item The pair $(\CC,\CC)$ has the LFP.
\item The cone $\CC$ is isomorphic to a direct sum of Lorentz cones.
\end{enumerate}
\end{thm}

In particular, a proper cone $\CC$ is a direct sum of Lorentz cones if and only if the pairs $(\CC,\CC')$ have the LFP for any proper cone $\CC'$. Restricting to the family of \emph{polyhedral cones}, i.e., cones with a finite number of extremal rays, we have a stronger result. 

\newpage

\begin{thm}\label{thm:LFPPolyhedral}
Let $\CC_1$ and $\CC_2$ be proper polyhedral cones. The following are equivalent
\begin{enumerate}
\item The pair $(\CC_1,\CC_2)$ has the LFP.
\item Either $\CC_1$ or $\CC_2$ is classical.
\end{enumerate}
\end{thm}

Here, we call a cone \emph{classical} if it is proper in a $d$-dimensional vector space and has exactly $d$ extremal rays. Any base of a classical cone is a simplex. Our terminology is motivated by a connection to quantum information theory (see~\cite{aubrun2021entangleability}), where classical cones correspond to classical physical theories.

To find interesting examples of pairs with the LFP, we consider the special case where the first cone is the \emph{square-based cone} in $\R^3$ given by 
\[ \CC_{\square} = \{ (x,y,z) \in \R^3 \st |y|+|z| \leq x \} .\]
In this case, we have the following theorem:

\begin{thm} \label{thm:symmetric-cones-LFP}
If $\CC$ is a symmetric cone, then $(\CC_\square,\CC)$ has the LFP.
\end{thm}

Here, a proper cone $\CC$ is called \emph{symmetric} if it is selfdual and its automorphism group acts transitively on the interior $\inter(\CC)$. Vinberg's classification theorem~\cite{vinberg1963theory} shows that a proper cone $\CC$ is symmetric if and only if it is a direct sum of cones from the following families:
\begin{itemize}
\item Lorentz cones $\L_n$ for $n\in\N$.
\item Cones $\PSD_n(\R),\PSD_n(\C),\PSD_n(\quat)$ for $n\in\N$ of positive semidefinite $n\times n$ matrices over the real numbers, the complex numbers and the quaternions, respectively.
\item The cone $\PSD_3(\Oct)$ of positive semidefinite $3\times 3$ matrices over the octonions.
\end{itemize}
We have the isomorphisms $\PSD_2(\R) \simeq \L_3$, $\PSD_2(\C) \simeq \L_4$, $\PSD_2(\quat) \simeq \L_6$, and $\PSD_2(\Oct) \simeq \L_{10}$.
Classical cones are included in this list since they are direct sums of $1$-dimensional cones. 

Lorentzian cones are the only examples of symmetric cones that are strictly convex, i.e., such that every non-zero boundary point generates an extremal ray. In this regard, we have the following partial converse to Theorem \ref{thm:symmetric-cones-LFP} in the case where $\CC$ is strictly convex. 

\begin{thm} \label{theorem:strictly-convex}
Let $\CC$ be a proper strictly convex cone such that $(\CC_\square,\CC)$ has the LFP. Then $\CC$ is Lorentzian.
\end{thm}

Finally, we note that the cone $\PSD_3(\R)$ residing in the $6$-dimensional space $\Her_3(\R)$ is the smallest symmetric cone that is not isomorphic to a direct sum of Lorentz cones. For cones in lower dimensions we have the following partial converse of Theorem \ref{thm:symmetric-cones-LFP}. 

\begin{thm} \label{theorem:dimension5}
Let $\CC$ be a proper cone with $\dim(\CC)\leq 5$ such that $(\CC_\square,\CC)$ has the LFP. Then $\CC$ is symmetric. 
\end{thm}

To make Theorem \ref{theorem:dimension5} more explicit, we state the complete list of symmetric cones in small dimensions. A $d$-dimensional proper cone $\CC$ for $d\leq 5$ is symmetric if 
\begin{itemize}
\item $d=1$ and $\CC$ is isomorphic to $\R_+$.
\item $d=2$ and $\CC$ is isomorphic to $\R_+^2$.
\item $d=3$ and $\CC$ is isomorphic to $\R_+^3$ or $\L_3$.
\item $d=4$ and $\CC$ is isomorphic to $\R_+^4$ or $\L_3 \oplus \R_+$ or $\L_4$.
\item $d=5$ and $\CC$ is isomorphic to $\R_+^5$ or $\L_3 \oplus \R_+^2$ or $\L_4 \oplus \R_+$ or $\L_5$.
\end{itemize}

A non-trivial example excluded by Theorem \ref{theorem:dimension5} is the so-called Vinberg cone $\CC_{\text{Vin}}\subseteq \R^5$ (see~\cite[Chapter I, Exercise 10]{faraut1994analysis}) given by 
\[
\CC_{\text{Vin}} = \lset (a,b_1,b_2,c_1,c_2)\in \R^5~:~ a\geq 0, ac_1\geq b^2_1, ac_2\geq b^2_2\rset .
\]
Geometrically, this cone corresponds to two $\L_3\simeq \PSD_2(\R)$ ``glued onto a ray", and it can be shown that this cone is homogeneous but not self-dual. As $\CC_{\text{Vin}}$ is not symmetric, Theorem \ref{theorem:dimension5} shows that $(\CC_\square,\CC_{\text{Vin}})$ does not have the LFP. 

\section{Preliminaries}

We will start by reviewing some notation and well-known preliminaries on proper cones and positive maps.  

\subsection{Proper cones and their geometry} Let $V$ denote a real finite-dimensional vector space and $\CC\subseteq V$ a convex cone. We say that $\CC$ is \emph{pointed} if it does not contain a line, i.e., if $\CC\cap (-\CC)=\lset 0\rset$, and we say that $\CC$ is \emph{generating} if it spans the entire ambient space $V$. Equivalently, a convex cone is generating if its interior is not empty. If $\CC$ is pointed, generating and closed, then we call it a \emph{proper cone}. If $\CC\subseteq V$ is a proper cone, then its dual
\[
\CC^*:=\lset f:V\ra \R~:~f(x)\geq 0, x\in\CC\rset\subseteq V^* ,
\]
is a proper cone as well. Throughout this article, we will only consider cones in finite-dimensional vector spaces and mostly focus on proper cones.

We say that a convex cone $\CC$ is a \emph{direct sum} of convex cones $\CC_1$ and $\CC_2$ if $\mathspan(\CC_1)\cap \mathspan(\CC_2)=\lset 0\rset$ and $\CC=\CC_1 + \CC_2$. We write $\CC=\CC_1\oplus \CC_2$ in this case. A convex cone $\CC$ is called \emph{decomposable} if there exist non-zero convex cones $\CC_1$ and~$\CC_2$ such that $\CC=\CC_1\oplus \CC_2$, and \emph{indecomposable} otherwise.

A convex cone $F\subseteq \CC$ is called a \emph{face} of $\CC$ if for any $x\in F$ and $y,z\in \CC$ with $x=y+z$, we have $y,z\in F$. We will sometimes need the following equivalent characterization of faces: A subset $F\subseteq \CC$ is a face of $\CC$ if for any $x\in F$ and $y\in \CC$ such that for some $\e>0$ we have $x-\e y\in\CC$ it follows that $y\in F$. The non-trivial faces of a cone $\CC$ are in $1$-to-$1$ correspondence with faces of a base of the cone. An \emph{extremal ray} of a cone $\CC$ is a $1$-dimensional face and we say that $x\in \CC$ is \emph{extremal} if $\R^+x$ is an extremal ray (note that $0\in \R^+$). For any $x\in \CC$ we set $\face(x)\subseteq \CC$ to be the smallest face of $\CC$ containing $x$. Note that $\face(x)=\R^+x$ if and only if $x$ is extremal and $\face(x)=\CC$ if and only if $x\in \text{int}(\CC)$. A proper cone $\CC\subseteq V$ is called
\begin{itemize}
\item \emph{classical}, if it has exactly $\dim(V)$ extremal rays.
\item \emph{polyhedral}, if it has finitely many extremal rays.
\item \emph{strictly convex}, if $\R^+x$ is an extremal ray for any nonzero boundary point $x\in \partial \CC$.
\end{itemize}

We also need a simple lemma about strictly convex cones.

\begin{lem} \label{lemma:basic-strictlyconvex}
Let $\CC \subseteq V$ be a strictly convex proper cone. Then
\begin{enumerate}
\item[(i)] Consider $z \in V$ such that $z \not\in \CC$ and $-z \not \in \CC$. There exist $x,y$ extremal in $\CC$ and non-collinear such that $z=x-y$.
\item[(ii)] Consider $a,b,c$ extremal and pairwise non-collinear in $\CC$. There exist a scalar $\lambda>0$ and an element~$d$, extremal in $\CC$ and non-collinear with $a$, $b$, $c$, such that $a+b=\lambda c +d$.
\end{enumerate}
\end{lem}

\begin{proof}
(i) The result is elementary if $\CC$ is $2$-dimensional (every $2$-dimensional cone is isomorphic to $\R_+^2$). For the general case, write $z = u-v$ for $u,v \in \CC$ and observe that $W \coloneqq \mathspan (u,v)$ has dimension $2$. By the $2$-dimensional case, there exist $x,y$ extremal in $\CC \cap W$ and non-collinear such that $z=x-y$. Since $\CC$ is strictly convex, $x$ and $y$ are extremal in $\CC$.

(ii) For $t \geq 0$, set $d(t)=a+b - t c \in V$. There exists $t > 0$ such that $d(t) \not \in \CC$ (otherwise, we would have $-c = \lim_{t \to + \iy} \frac{d(t)}{t} \in \CC$, a contradiction. Since $d(0) \in \inter(\CC)$, there exists $\lambda>0$ such that $d(\lambda) \in \partial \CC$. Since $d(\lambda) \neq 0$ and $\CC$ is strictly convex, $d(\lambda)$ is extremal in $\CC$.
 \end{proof}

\subsection{Retracts}

Let $\CC \subseteq V$ and $\CC' \subseteq V'$ be proper cones. We say that $\CC'$ is a \emph{retract} of $\CC$ if there exist positive maps $\alpha \in \Pos(\CC',\CC)$ and $\beta \in \Pos(\CC,\CC')$ such that $\beta \alpha = \ident_{V'}$. It is easy to see that both $\CC$ and $\CC'$ are retracts of the direct sum $\CC\oplus \CC'$. Moreover, we have the following lemma, which appeared in~\cite[Proposition S5]{aubrun2019universal} and which we reprove here for convenience.

\begin{lem} \label{retract-NonClassical}
Every proper non-classical polyhedral cone $\CC$ has a proper non-classical polyhedral retract $\CC'\subseteq \R^3$.
\end{lem}
\begin{proof}
We use the fact that any polytope of dimension $\geq 3$ that is both simple and simplicial is a simplex (see \cite[Theorem 12.19]{Bronsted12}), or equivalently that any non-classical polyhedral cone of dimension $\geq 4$ has a nonclassical facet or a nonclassical vertex figure. We now note that facets and vertex-figures of the proper polyhedral cone $\CC$ are retracts of~$\CC$ (see~\cite[Proposition 6.16]{de2020tensor}). If $\CC$ is not classical, we can therefore iteratively select non-classical facets and vertex-figures until we arrive at a non-classical cone $\CC'\subseteq \R^3$ that is a retract of $\CC$.
\end{proof}

By \cite[Example 6.15~(g)]{de2020tensor} the cone $\PSD_n(\C)$ is a retract of $\PSD_{2n}(\R)$. In a similar way, it can be shown that $\PSD_n(\quat)$ is a retract of $\PSD_{2n}(\C)$, and hence of $\PSD_{4n}(\R)$. The cone $\PSD_3(\Oct)$ is \emph{not} a retract of $\PSD_{n}(\R)$ for any $n\in\N$. This can be proved using standard techniques from the theory of Jordan algebras and the fact that the Jordan algebra of Hermitian $3\times 3$ matrices over the octonions, known as the Albert algebra, is exceptional (see for example~\cite[Corollary 2.8.5]{hanche1984jordan}).

\subsection{Positive maps and Lorentz-factorizations} Let $\CC_1\subseteq V_1$ and $\CC_2\subseteq V_2$ denote proper cones. We say that a linear map $T:V_1\ra V_2$ is \emph{$(\CC_1,\CC_2)$-positive} if $T(\CC_1)\subseteq \CC_2$. We will simply call a map \emph{positive} if it is clear from context which cones are involved. The set $\Pos(\CC_1,\CC_2)$ of all $(\CC_1,\CC_2)$-positive maps forms a proper cone inside the vector space of linear maps from $V_1$ to $V_2$. We make repeated use of the following lemma, which appears for example in \cite[Corollary 4.5]{HFP76}.

\begin{lem} \label{lemma:rank-not-2}
Let $\CC_1$, $\CC_2$ be proper cones and $T$ be extremal in $\Pos(\CC_1,\CC_2)$. Then $\rank (T) = 1$ or $\rank (T) \geq 3$. 
\end{lem}

For proper cones $\CC_1$ and $\CC_2$ we define the cone of \emph{Lorentz factorizable maps} as 
\[ \LorFact(\CC_1,\CC_2) = \conv \{\beta \alpha \st \alpha \in \Pos(\CC_1,\L), \ \beta \in \Pos(\L,\CC_2), \ \L \textnormal{ a Lorentzian cone} \} \]
It has been shown in~\cite{LPMH25} that $\LorFact(\CC_1,\CC_2) \subseteq \mathcal{L}(V_1,V_2)$ is a proper cone. Every Lorentz factorizable map is obviously positive. We will need the following properties of Lorentzian cones:

\newpage

\begin{lem}\label{lem:quotientLorentzCone}
Let $\L\subseteq V$ denote a Lorentzian cone, $W\subseteq V$ a subspace, and $\pi:V\ra V/W$ the canonical quotient map. We have the following:
\begin{enumerate}
\item If $W\cap \inter(\L)\neq \emptyset$, then $W \cap \L$ is Lorentzian.
\item If $W\cap \L=\lset 0\rset$, then $\pi(\L)$ is Lorentzian.
\end{enumerate}
\end{lem}

\begin{proof}
The first statement follows immediately from \cite[\S 5, Lemma 27]{o1983semi}. To prove the second statement, consider a quadratic form $q$ of signature $(1,\dim(V)-1)$ such that $q(x) \geq 0$ if and only if $x \in \pm \L \coloneqq \L \cup (-\L)$. Let $g$ be the associated bilinear form. Define
\[
W^\perp = \lset v\in V~:~g(v,w)=0 \text{ for all }w\in W\rset .
\]
Since $W\cap \L=\lset 0\rset$, the restriction $q|_{W}$ is negative definite and in particular non-degenerate. 
By \cite[\S 2, Lemma 23]{o1983semi}, this implies that $V=W\oplus W^{\perp}$. Let $\alpha : V/W \to W^\perp$ the associated isomorphism. A vector $y \in W^\perp$ belongs to $\pm\alpha(\pi(\L))$ if and only if there exists $x \in W$ such that $q(x+y) \geq 0$, or equivalently $q(x)+q(y) \geq 0$. Since $q \leq 0$ on $W$, we conclude that $\pm\alpha(\pi(\L))=W^\perp \cap \pm\L$ and thus $\alpha(\pi(\L))=W^\perp \cap \L$. Since $W^{\perp}\cap \inter(\L)\neq \emptyset$, the first part of the lemma implies that $W^\perp \cap \L$ is Lorentzian, hence $\pi(\L)$ is also Lorentzian. 
\end{proof}

The following lemma will be useful when studying Lorentz-factorizable maps:

\begin{lem} \label{lemma:lorfact:inj-surj}
Let $\CC_1$, $\CC_2$ be proper cones, $\L$ be a Lorentzian cone and $\alpha \in \Pos(\CC_1,\L)$, $\beta \in \Pos(\L,\CC_2)$ be positive linear maps. There exists a Lorentzian cone $\L'$, a surjective linear map $\alpha' \in \Pos(\CC_1,\L')$ and an injective linear map $\beta' \in \Pos(\L',\CC_2)$ such that $\beta'\alpha'=\beta\alpha$.
\end{lem}

\begin{proof}
Let $P=\beta\alpha$, and note that the statement is clear if $\rank(P)=1$. In the following, we assume that $\rank(P)\geq 2$. Let $E$ be the range of $\alpha$. By Lemma~\ref{lem:quotientLorentzCone}, the cone $\L_0\coloneq \L \cap E$ is Lorentzian. Now, define $\alpha_0\in \Pos(\CC_1,\L_0)$ by formally restricting the codomain of $\alpha$ and let $\beta_0 \in \Pos(\L_0,\CC_2)$ be $\beta|_{E}$. Then, we have $P=\beta_0\alpha_0$ and $\alpha_0$ is surjective. 

Let $F\subseteq E$ denote the kernel of $\beta_0$. We claim that $F \cap \L_0 = \{0\}$. Assume by contradiction that $F \cap \L_0$ contains a nonzero element $x$. We necessarily have $x \in \partial \L_0$. Since $\rank \beta_0 \geq 2$, there exist $\varphi_1,\varphi_2 \in \CC_2^*$ such that $\varphi_1\beta_0$ and $\varphi_2\beta_0$ are non-collinear. This is a contradiction since there is up to scaling a unique element in $\L_0^*$ which vanishes on $x$.

Let $\pi:E\to E/F$ denote the canonical quotient map.
By Lemma \ref{lem:quotientLorentzCone}, the cone $\L'\coloneq\pi(\L_0)$ is Lorentzian. Finally, we set $\alpha'=\pi \alpha_0\in \Pos(\CC_1,\L')$ and $\beta'\in \Pos(\L',\CC_2)$ by setting $\beta'([x]_{F})=\beta_0(x)$. Note that $\beta'$ is well-defined and injective, and $\alpha'$ is surjective as a composition of surjective maps. It is easy to verify that $P=\beta'\alpha'$, finishing the proof.
\end{proof}

\subsection{The square-base cone}

The simplest example of a non-classical polyhedral cone is the \emph{square-based cone} given by
\[ \CC_{\square} = \{ (x,y,z) \in \R^3 \st |y|+|z| \leq x \} \]
In the early article \cite{namioka1969tensor} it was shown that a proper cone $\CC$ is non-classical if and only if $\Pos(\CC_\square,\CC)$ contains an extremal element of rank $>1$.
The cone $\CC_{\square}$ has four extremal rays which are generated by the following vectors
\[ v_{+0} = (1,1,0), \ v_{-0} = (1,-1,0), \ v_{0+} = (1,0,1), \ v_{0-} = (1,0,-1) \]
These vectors satisfy the relation $v_{+0}+v_{-0}=v_{0+}+v_{0-}$. Suppose that $V$ is a vector space containing elements $x_{+0}$, $x_{-0}$, $x_{0+}$, $x_{0-}$ such that $x_{+0}+x_{-0}=x_{0+}+x_{0-}$. We denote by 
\[ \Sqmap{x_{+0}}{x_{0+}}{x_{0-}}{x_{-0}} \]
the linear map $T : \R^3 \to V$ satisfying the conditions
\begin{equation} \label{eq:map-from-square} T(v_{+0})=x_{+0},\ T(v_{-0})=x_{-0},\ T(v_{0+})=x_{0+},\ T(v_{0-})=x_{0-} 
\end{equation}

We have the following:

\begin{lem} \label{lemma:rescaling}
Let $x_1,x_2,x_3,x_4$ be nonzero elements of a proper cone $\CC$ in convex position, i.e., for every $i\in\lset 1,2,3,4\rset$ we have $x_i \not \in \cone \{x_j \,:\, j\neq i\}$. If $\mathspan\lset x_1,x_2,x_3,x_4\rset$ is $3$-dimensional, then there exist $\lambda_i>0$ such that one of the following equations holds
\begin{gather*}
\lambda_1 x_1 + \lambda_2 x_2 = \lambda_3 x_3 + \lambda_4 x_4, \\
\lambda_1 x_1 + \lambda_3 x_3 = \lambda_2 x_2 + \lambda_4 x_4, \\
\lambda_1 x_1 + \lambda_4 x_4 = \lambda_2 x_2 + \lambda_3 x_3.
\end{gather*}
In particular, any $3$-dimensional proper cone with exactly four extremal rays is isomorphic to $\CC_\square$.
\end{lem}

\begin{proof} By assumption, there are $\alpha_1,\alpha_2,\alpha_3,\alpha_4\in \R$ not all zero, such that
\[
\alpha_1 x_1 + \alpha_2 x_2 + \alpha_3 x_3 + \alpha_4 x_4 = 0
\]
As $x_1,x_2,x_3,x_4$ are in convex position, at least two of the coefficients $\alpha_1,\alpha_2,\alpha_3,\alpha_4$ are strictly positive and at least two are strictly negative. Then, the previous relation is equivalent to one of the relations stated in the lemma.
\end{proof}

We will need the following criterion for whether a map in $\Pos(\CC_\square,\CC)$ is extremal. Similar criteria can be found in~\cite{banacki2021edge}, and an equivalent characterization of extremal maps in $\Pos(\CC_\square,\CC)$ has been shown in~\cite[Proposition 5]{jenvcova2018incompatible}.  

\begin{prop} \label{prop:extremality-conditions}
Let $\CC \subseteq V$ be a proper cone and $x_1,x_2,x_3,x_4 \in \CC$ such that $x_1+x_4=x_2+x_3$. The map $T=\sqmap{x_1}{x_2}{x_3}{x_4}$ is extremal in $\Pos(\CC_\square,\CC)$ if and only if one of the following occurs.
\begin{enumerate}
\item[\textup{(i)}] The map $T$ has rank $1$, and there exists an element $y$ extremal in $\CC$ such that
\[ T = \Sqmap{y}{y}{0}{0} \textnormal{ or } T = \Sqmap{y}{0}{y}{0} \textnormal{ or } T = \Sqmap{0}{0}{y}{y} \textnormal{ or } T = \Sqmap{0}{y}{0}{y}\]
\item[\textup{(ii)}] The map $T$ has rank $3$, and the linear subspaces $E_i \coloneqq \mathspan(\face(x_i))$ satisfy the following relations
\begin{gather*}
E_i \cap E_j =\{0\} \textnormal{ for every } 1 \leq i <j \leq 4\\
(E_1+E_4) \cap (E_2 + E_3)  = \R^+ (x_1+x_4)
\end{gather*}
\end{enumerate}
\end{prop}

The proof uses the following elementary observation. Consider elements  $x \in \CC$ and $y \in V$. Then $y \in \mathspan (\face (x))$ if and only if $x+\e y \in \CC$ and $x-\e y \in \CC$ for sufficiently small real numbers $\e>0$.

\begin{proof}
By Lemma \ref{lemma:rank-not-2}, the rank of an extremal map is either $1$ or $3$. The characterization of extremal maps with rank $1$ follows from the form of the four extremal rays in the dual cone $\CC^*_\square$ and it yields the condition in case (i). 

Suppose now that $\rank (T) =3$. We will show that the failure of any of the conditions in (ii) imply that $T$ is not extremal. Assume $E_1 \cap E_4$ contains a nonzero element $y$. The map $S = \sqmap{y}{0}{0}{-y}$ is not proportional to $T$ since \mbox{$\rank(S)=1$}. Let $\e>0$ be small enough such that $x_1 \pm \e y$ and $x_4 \pm \e y$ belong to $\CC$. The maps $T+\e S$ and $T-\e S$ are positive and therefore $T$ is not extremal. Assume that $E_1 \cap E_2$ contains a nonzero element $z$. The map $R=\sqmap{z}{z}{0}{0}$ is not proportional to~$T$. Let $\e>0$ be small enough such that $x_1 \pm \e z$ and $x_2 \pm \e z$ belong to $\CC$. The maps $T+\e R$ and $T- \e R$ are positive and therefore $T$ is not extremal. The same arguments applied to other pairs of indices shows that $T$ is not extremal whenever $E_i \cap E_j \neq \{0\}$  for some $1 \leq i < j \leq 4$.

Set $x:=x_1+x_4=x_2+x_3$ and assume that $(E_1+E_4) \cap (E_2 + E_3)$ contains an element $w$ not proportional to $x$. There exist elements $w_i$ in $E_i$ ($1 \leq i \leq 4$) such that $w=w_1+w_4=w_2+w_3$ and the map $Q = \sqmap{w_1}{w_2}{w_3}{w_4}$ is not proportional to $T$.  Pick $\e>0$ small enough such that $x_i\pm \e w_i \in \CC$ for $1 \leq i \leq 4$. The maps $T+\e Q$ and $T-\e Q$ are positive and therefore $T$ is not extremal.

Finally, it remains to show that a positive map $T=\sqmap{x_1}{x_2}{x_3}{x_4}$ satisfying the conditions in (ii) is indeed extremal. Consider a positive map $S=\sqmap{y_1}{y_2}{y_3}{y_4}$ such that $T-S$ is positive as well. In this case, $y_i\in E_i$ for any $i\in\lset 1,2,3,4\rset$ and we conclude from (ii) that 
\[
y_1+y_4=y_2+y_3=\lambda (x_1+x_4)=\lambda (x_2+x_3) ,
\]
for some $\lambda>0$. In particular, we have $y_1-\lambda x_1 = \lambda x_4 - y_4$. Since $y_1-\lambda x_1\in E_1$ and $\lambda x_4-y_4\in E_4$ we conclude from (ii) that $y_1 = \lambda x_1$ and $y_4=\lambda x_4$. Similarly, we also find that $y_2 = \lambda x_2$ and $y_3=\lambda x_3$ and we conclude that the map $T$ is extremal.
\end{proof}

Here is a consequence of Proposition \ref{prop:extremality-conditions} which could also be proved directly.

\begin{cor} \label{cor:extremality-conditions}
Let $x_1$, $x_2$, $x_3$, $x_4$ be pairwise non-collinear, extremal in $\CC$ and such that $x_1+x_4=x_2+x_3$. Then the map $T \coloneqq \sqmap{x_1}{x_2}{x_3}{x_4}$ is extremal in $\Pos(\CC_\square,\CC)$.
\end{cor}

\begin{proof}
The hypotheses imply that $\rank(T)=3$ and $\mathspan(\face(x_i))=\mathspan(x_i)$. Moreover, if scalars $\alpha_i$ satisfy $\alpha_1 x_1 + \alpha_4 x_4 = \alpha_2 x_2 + \alpha_3 x_3$ then necessarily $\alpha_1=\alpha_2=\alpha_3=\alpha_4$, so that condition (ii) from Proposition \ref{prop:extremality-conditions} is satisfied.
\end{proof}

\section{LFP pairs and Lorentz cone sandwiching}

\subsection{General properties and techniques} The following lemma collects some operations acting on pairs of proper cones, which preserve the LFP. 

\begin{prop}\label{prop:LFP-passesTo}
Let $\CC_1$, $\CC_2$ be proper cones such that $(\CC_1,\CC_2)$ has the LFP. Then, we have the following:
\begin{enumerate}
\item The pair $(\CC^*_2,\CC^*_1)$ has the LFP.
\item If $F$ is a face of $\CC_2$, then $(\CC_1,F)$ has the LFP.
\item \label{item: retract} If $\CC'_1$ is a retract of $\CC_1$ and $\CC'_2$ is a retract of $\CC_2$, then $(\CC'_1,\CC'_2)$ has the LFP.
\item \label{item: direct_sum} If $\CC_1 = \CC'_1\oplus \CC''_1$ and $\CC_2 = \CC'_2\oplus \CC''_2$, then all pairs $(\CC'_1,\CC'_2),(\CC''_1,\CC'_2),(\CC'_1,\CC''_2)$ and $(\CC''_1,\CC''_2)$ have the LFP.
\end{enumerate}
\end{prop}

\begin{proof}
To prove (1) consider a positive map $T\in \Pos(\CC^*_2,\CC^*_1)$ and note that its adjoint $T^*\in \Pos(\CC_1,\CC_2) $ is Lorentz-factorizable. Taking the dual of this factorization, using that Lorentz cones are selfdual, produces the desired factorization for $T$. 

For (2) consider $T \in \Pos(\CC_1,F)$. Since $(\CC_1,\CC_2)$ has the LFP and using Lemma~\ref{lemma:lorfact:inj-surj}, there exist an integer $N \in \N$ and for every $1 \leq i \leq N$ a Lorentzian cone $\L^{i}$ in a vector space $V^i$, a surjective map $\alpha_i \in \Pos(\CC_1,\L^i)$ and an injective map $\beta_i \in \Pos (\L^i,\CC_2)$ such that $T = \sum^N_{i=1} \beta_i \alpha_i$. Consider an index  $i\in \{ 1,\ldots ,N\}$ and  an element $y \in V^i$. Since $\alpha_i$ is surjective and~$\CC_1$ is generating, there exists elements $x_+$, $x_-$ in $\CC_1$ such that $y=\alpha_i(x_+-x_-)$. For every $x \in \CC_1$, since $T(x)$ belongs to the face $F$, we have $\beta_i(\alpha_i(x)) \in F$. It follows that $\beta_i(y) = \beta_i(\alpha_i(x_+))-\beta_i(\alpha_i(x_-))$ belongs to $F-F$. This proves that the range of $\beta_i$ lies in $\mathspan (F)$. We may thus consider $\beta_i$ as an element of $\Pos(L^i,F)$. It follows that $T \in \LorFact(\CC_1,F)$.

For (3) let $\alpha_1\in \Pos(\CC'_1,\CC_1)$, $\beta_1\in \Pos(\CC_1,\CC'_1)$ and $\alpha_2\in \Pos(\CC'_2,\CC_2)$, $\beta_2\in \Pos(\CC_2,\CC'_2)$ be such that $\ident_{V'_1}=\beta_1\alpha_1$ and $\ident_{V'_2}=\beta_2\alpha_2$. For any $T'\in \Pos(\CC'_1,\CC'_2)$ we can consider $T = \alpha_2 T'\beta_1\in \Pos(\CC_1,\CC_2)$. If $(\CC_1,\CC_2)$ has the LFP, then $T$ factors through Lorentz cones and the same holds for $T'=\beta_2 T \alpha_1$.

Statement (4) follows from the previous statement and the fact that $\CC_1'$ and $\CC_1''$ are retracts of $\CC_1'\oplus \CC_1''$, and that $\CC_2'$ and $\CC_2''$ are retracts of $\CC_2'\oplus \CC_2''$.
\end{proof}

Next, we prove an easy lemma, which constitutes our main tool to show that certain positive maps are not Lorentz-factorizable.

\begin{lem}[Lorentz-cone sandwiching] \label{lemma:lorentz-extreme}
Let $\CC_1,\CC_2$ be proper cones. Let $T\in \Pos(\CC_1,\CC_2)$ be extremal and $k \coloneqq \rank (T)$. The map $T$ is Lorentz-factorizable if and only if there exists a $k$-dimensional Lorentzian cone $\L\subseteq \range(T)$ such that $T(\CC_1) \subseteq \L \subseteq \CC_2$.
\end{lem}

\begin{proof}
Let $T\neq 0$ be extremal and Lorentz-factorizable. There exists an integer $N \geq 1$ and for every $1 \leq i \leq N$ a Lorentzian cone $\L^{i}$, a surjective map $\alpha_i \in \Pos(\CC_1,\L^i)$ and an injective map $\beta_i \in \Pos (\L^i,\CC_2)$ such that $T = \sum^N_{i=1} \beta_i \alpha_i$. Since~$T$ is extremal, it is proportional to $\beta_1\alpha_1$. The cone $\L \coloneqq \beta_1(\L^1)$ has the desired properties. The converse is clear.
\end{proof}

In the case of maps from the square-based cone, we have the following specializations of the previous lemma, which we will use heavily later on. 

\begin{lem}
\label{lemma:ellipse-sandwiching}
Let $\CC$ be a proper cone such that $(\CC_\square,\CC)$ has the LFP. Let $x_1$, $x_2$, $x_3$, $x_4$ in $\CC$ such that $x_1+x_4=x_2+x_3$ and such the map $T \coloneqq \sqmap{x_1}{x_2}{x_3}{x_4}$ has rank $3$ and is extremal in $\Pos(\CC_\square,\CC)$. For every $1 \leq i < j \leq 4$, the cone $T(\CC_\square)$ is contained in $\face(x_i+x_j)$.
\end{lem}

\begin{proof}
By Lemma \ref{lemma:lorentz-extreme}, there exists a $3$-dimensional Lorentzian cone $\L$ containing the points $x_i$ and such that $\L \subseteq \CC$. Since $\L$ is strictly convex, we have $x_i+x_j \in  \relint (\L)$ and therefore $\L \subseteq \face(x_i+x_j)$.
\end{proof}

\begin{lem} \label{lemma:face_with_no_EB}
Let $\CC$ be a proper cone such that $(\CC_\square,\CC)$ has the LFP. Let $x_1$, $x_2$, $x_3$, $x_4$ in $\CC$ such that $x_1+x_4=x_2+x_3$. Assume that $x_1$ and $x_2$ are extremal in $\CC$. Denote $F_i = \face(x_i)$. If
\begin{equation} \label{eq:noEBmap}
F_1 \cap F_2 = F_1 \cap F_3 = F_2 \cap F_4 = F_3 \cap F_4 = \{0\} \end{equation}
then $x_3$ and $x_4$ belong to $\face (x_1+x_2)$.
\end{lem}

\begin{proof}
The map $T = \sqmap{x_1}{x_2}{x_3}{x_4}$ belongs to the face $G$ of $\Pos(\CC_\square,\CC)$ defined as
\[ G = \left\{ \Sqmap{a_1}{a_2}{a_3}{a_4} \st a_1+a_4=a_2+a_3,\ a_i \in F_i \textnormal{ for } 1 \leq i \leq 4 \right\} .\]
Consider a map $P=\sqmap{a_1}{a_2}{a_3}{a_4}$ which is extremal in $G$ and hence also in $\Pos(\CC_\square,\CC)$. Equation \eqref{eq:noEBmap} implies that $G$ does not contain any rank $1$ map ; it follows that $\rank(P)=3$. By Lemma~\ref{lemma:ellipse-sandwiching}, we have $a_3,a_4 \in \face (a_1+a_2) = \face (x_1+x_2)$. The result follows by decomposing $T$ as a sum of extremal maps $T=\sum_k P_k$ with $P_k\in G$.
\end{proof}

To illustrate how to use this lemma, we will now prove two of our main results.

\subsection{Proofs of Theorem \ref{thm:SameCone} and Theorem \ref{thm:LFPPolyhedral}}

We start with the first result concerning the case where $\CC_1=\CC_2=\CC$ for some proper cone $\CC$.

\begin{theoremrestated}{thm:SameCone}[restated]
Let $\CC$ be a proper cone. The following are equivalent
\begin{enumerate}
\item The pair $(\CC,\CC)$ has the LFP.
\item The cone $\CC$ is isomorphic to a direct sum of Lorentz cones.
\end{enumerate}
\end{theoremrestated}

\begin{proof}
By statement \eqref{item: direct_sum} in Proposition~\ref{prop:LFP-passesTo} we may assume without loss of generality that $\CC\subseteq V$ is indecomposable. By \cite[Theorem 3.3]{loewy1975indecomposable} the identity map $\ident_V$ is extremal in $\Pos(\CC,\CC)$. If $\CC$ has the LFP, then by Lemma \ref{lemma:lorentz-extreme} there is a Lorentzian cone $\L$ such that $\CC=\ident_V(\CC)\subseteq \L \subseteq \CC$ and hence $\CC=\L$. 
\end{proof}

Next, we identify the pairs $(\CC_1,\CC_2)$ of polyhedral cones which have the LFP.

\begin{theoremrestated}{thm:LFPPolyhedral}[restated]
Let $\CC_1$ and $\CC_2$ be proper polyhedral cones. The following are equivalent
\begin{enumerate}
\item The pair $(\CC_1,\CC_2)$ has the LFP.
\item Either $\CC_1$ or $\CC_2$ is classical.
\end{enumerate}
\end{theoremrestated}

\begin{proof}
Clearly, the pair $(\CC_1,\CC_2)$ has the LFP if either $\CC_1$ or $\CC_2$ is classical. To prove the other direction, we assume that $\CC_1$ and $\CC_2$ are proper polyhedral cones that are not classical. By Lemma \ref{retract-NonClassical}, both of these cones have non-classical $3$-dimensional retracts. By statement \eqref{item: retract} in Proposition~\ref{prop:LFP-passesTo}, it is therefore enough to show that any pair of proper non-classical $3$-dimensional polyhedral cones does not have the LFP.

Let $\CC_1\subseteq V_1$ and $\CC_2\subseteq V_2$ be proper non-classical $3$-dimensional polyhedral cones. There exist pairwise non-collinear elements $x_1,x_2,x_3,x_4$ extremal in $\CC_1$ such that $F_{12}:=\face(x_1+x_2)$, $F_{23}:=\face(x_2+x_3)$ and $F_{34}:=\face(x_3+x_4)$ are 2-dimensional faces of $\CC_1$. Similarly, there exist pairwise non-collinear elements $y_1,y_2,y_3,y_4$ extremal in $\CC_2$ such that $G_{12}:=\face(y_1+y_2)$, $G_{23}:=\face(y_2+y_3)$ and $G_{34}:=\face(y_3+y_4)$ are 2-dimensional faces of $\CC_2$.
Next, we define the set 
\[
\mathcal{F} = \lset T\in\Pos(\CC_1,\CC_2)~:~T(x_1)\in G_{12},\ T(x_2)\in \R_+ y_2,\ T(x_3)\in \R_+ y_3,\ T(x_4)\in G_{34}\rset .
\]
The set $\mathcal{F}$ is a face of the cone of positive maps $\Pos(\CC_1,\CC_2)$. To finish the proof we will find an extremal map in $\mathcal{F}$ that is not Lorentz-factorizable.

Consider extremal $f_{12},f_{23},f_{34}\in \CC_1^*$ such that $\ker( f_{12}) \cap \CC_1 = F_{12}$, $\ker(f_{23}) \cap \CC_1 = F_{23}$ and $\ker( f_{34} )\cap  \CC_1 = F_{34}$. Observe that any extremal element of $\CC^*_1$ vanishing on $x_2$ is proportional to either $f_{12}$ or $f_{23}$, and any extremal element of $\CC_1^*$ vanishing on~$x_3$ is proportional to either $f_{23}$ or $f_{34}$. Let $S\in \mathcal{F}$ denote an extremal map with $\rank(S)=1$. Then, we have $S(x)=f(x)y$ for a non-zero extremal $f\in \CC^*_1$ and a non-zero extremal $y\in \CC_2$. Since $y_2$ and $y_3$ are linearly independent, we have $f(x_2)=0$ or $f(x_3)=0$. If $f(x_2)=f(x_3)=0$, then we have $f(x_1)>0$ and $f(x_4)>0$ and hence $y\in G_{12}\cap G_{34}$ leading to a contradiction as $G_{12}\cap G_{34}=\lset 0\rset$. There are only two possibilities: Either $f(x_2)=0$ and $f(x_3)>0$, in which case $f$ must be proportional to $f_{12}$ and thus $y$ is proportional to $y_3$. Or, we have $f(x_2)>0$ and $f(x_3)=0$, in which case~$f$ must be proportional to $f_{34}$ and thus $y$ is proportional to $y_2$. We conclude that $S_1 = y_3 f_{12}$ and $ S_2 = y_2 f_{34}$ are the only extremal maps with rank $1$ in the face $\mathcal{F}$.

In the remaining part of the proof we will construct a map $T\in \mathcal{F}$ with rank~$3$. By the previous paragraph and Lemma \ref{lemma:rank-not-2}, the existence of such a map implies the existence of an extremal map $T'\in \mathcal{F}$ with $\rank(T')=3$. Such a map $T'$ is injective and there are $\alpha_2,\alpha_3 >0$ such that $T'(x_2)=\alpha_2y_2$ and $T'(x_3)=\alpha_3y_3$. By the choices made above, this implies
\[
y_2,y_3, y_2+y_3\in T'(\CC_1)\cap \partial \CC_2 ,
\]
and hence there cannot exist a $3$-dimensional Lorentzian cone $\L$ with $T'(\CC_1)\subseteq \L \subseteq \CC_2$ as such a cone would be strictly convex. By Lemma \ref{lemma:lorentz-extreme}, we conclude that the map $T'$ is not Lorentz-factorizable and $(\CC_1,\CC_2)$ does not have the LFP.

We finally construct a map $T\in\mathcal{F}$ with $\rank (T)=3$. Consider vectors $z,z'\in V_1$ such that $F_{12}\subseteq \cone \{x_2,z \}$, $F_{34}\subseteq \cone \{ x_3,z'\}$, and $\CC_1\subseteq \CC':=\cone\{z,x_2,x_3,z'\}$ (see Figure \ref{fig:Outertrapezoid}). As in Lemma~\ref{lemma:rescaling}, there is a linear map $T:V_1 \to V_2$ of rank $3$ such that $T(z) \in \R_+y_1$, $T(x_2) \in \R_+ y_2$, $T(x_3) \in \R_+ y_3$ and $T(z') \in \R_+y_4$. 

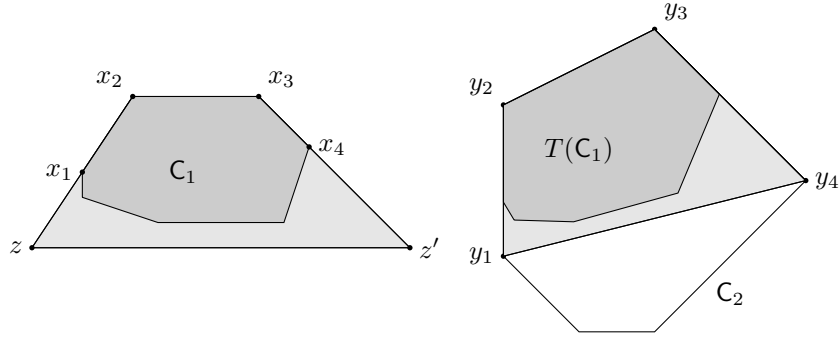
\begin{figure}[htbp]
\begin{minipage}{0.45\textwidth}
\begin{center}
\begin{tikzpicture}[scale=1/3]
    \coordinate (A) at (0,1);
    \coordinate (B) at (0,2);
    \coordinate (C) at (2,5);
    \coordinate (D) at (7,5);
    \coordinate (E) at (9,3);
    \coordinate (F) at (8,0);
    \coordinate (G) at (3,0);

    \coordinate (Z1) at (-2,-1);
    \coordinate (Z2) at (13,-1);

    \draw[fill={gray!20}] (C) -- (Z1) -- (Z2) -- (D) -- (C);
    \draw[fill={gray!40}] (A) -- (B) -- (C) -- (D) -- (E) -- (F) -- (G) -- (A);
        
    \fill (B) circle (3pt) node[left] {$x_1$};
    \fill (C) circle (3pt) node[above left] {$x_2$};
    \fill (D) circle (3pt) node[above right] {$x_3$};
    \fill (E) circle (3pt) node[right] {$x_4$};
    \node at (4,2) {$\CC_1$};

    \fill (Z1) circle (3pt) node[left] {$z$};
    \fill (Z2) circle (3pt) node[right] {$z'$};
    
    \draw (C) -- (Z1) -- (Z2) -- (D);
\end{tikzpicture}
\end{center}
\end{minipage}
\begin{minipage}{0.45\textwidth}
\begin{center}
\begin{tikzpicture}[scale=1/3]
    \coordinate (A) at (3*0.142,3*1.478);
    \coordinate (B) at (3*0,3*1.714);
    \coordinate (C) at (3*0,3*3);
    \coordinate (D) at (3*2,3*4);
    \coordinate (E) at (3*2.857,3*3.143);
    \coordinate (F) at (3*2.309,3*1.835);
    \coordinate (G) at (3*0.929,3*1.455);
    \coordinate (X) at (3*2,0);
    \coordinate (Y) at (3*1,0);

    \coordinate (Z1) at (3*0,3*1);
    \coordinate (Z2) at (3*4,3*2);

    \fill (Z1) circle (3pt) node[left] {$y_1$};
    \fill (C) circle (3pt) node[above left] {$y_2$};
    \fill (D) circle (3pt) node[above right] {$y_3$};
    \fill (Z2) circle (3pt) node[right] {$y_4$};

    \draw[fill={gray!20}] (C) -- (Z1) -- (Z2) -- (D) -- (C);
    \draw[fill={gray!40}] (A) -- (B) -- (C) -- (D) -- (E) -- (F) -- (G) -- (A);

    \node at (3,7.2) {$T(\CC_1)$};
    \node at (9,1.5) {$\CC_2$};

    
    \draw (C) -- (Z1) -- (Z2) -- (D);

    \draw (Z1) -- (C) -- (D) -- (Z2) -- (X) -- (Y) -- (Z1) ;
   
\end{tikzpicture}
\end{center}
\end{minipage}
\caption{Sketch of construction used in the proof of Theorem~\ref{thm:LFPPolyhedral}. All points are drawn up to scaling with a positive factor to make them coplanar.}\label{fig:Outertrapezoid}
\end{figure}

Since $T(\CC_1)\subseteq T(\CC')\subseteq \CC_2$, the map $T$ is positive. Since $T(x_1)\in T(\cone \{x_2,z\})\subseteq G_{12}$ and $T(x_4)\in T(\cone\{x_3,z'\})\subseteq G_{34}$, we conclude that $T\in\mathcal{F}$. 
\end{proof}

\section{LFP pairs with the square-based cone}

\subsection{When \texorpdfstring{$\CC$}{C} is strictly convex: Proof of Theorem \ref{theorem:strictly-convex}}

In planar geometry, three points $A,B,C$ and tangent lines at $A$ and $B$ determine at most one ellipse. To see this, we may apply a projective transformation to reduce to the case where the tangent lines at points $A$ and $B$ are parallel and both orthogonal to the line $(AB)$. This forces the segment $[AB]$ to be an axis of the ellipse, in which case the result is clear (see Figure~\ref{fig:ellipse}).

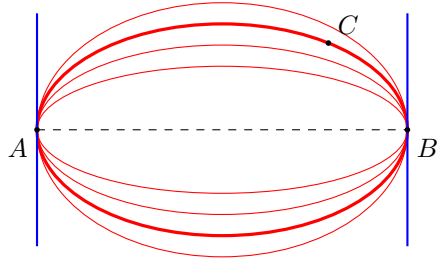
\begin{figure}[htbp]
\begin{center}
\begin{tikzpicture}[scale=0.7]

\def\a{3.5cm}
\def\b{2cm}

\draw[red, very thick] (0,0) ellipse [x radius=\a, y radius=\b];
\draw[red] (0,0) ellipse [x radius=\a, y radius=1.2*\b];
\draw[red] (0,0) ellipse [x radius=\a, y radius=0.8*\b];
\draw[red] (0,0) ellipse [x radius=\a, y radius=0.6*\b];

\coordinate (A) at (-3.5,0);
\coordinate (B) at ( 3.5,0);
\coordinate (C) at ({3.5*cos(55)},{2*sin(55)});

\draw[dashed] (A) -- (B);

\draw[blue, thick] (-3.5,-2.2) -- (-3.5,2.2);
\draw[blue, thick] ( 3.5,-2.2) -- ( 3.5,2.2);

\fill (A) circle (1.5pt) node[below left] {$A$};
\fill (B) circle (1.5pt) node[below right] {$B$};
\fill (C) circle (1.5pt) node[above right] {$C$};

\end{tikzpicture}
\end{center}
\caption{Among the family of ellipses having $[AB]$ as an axis, at most one passes through a point $C \neq A,B$}\label{fig:ellipse}
\end{figure}

\begin{lem} \label{lemma:unique-ellipse}
Let $x,y,z$ be linearly independent and extremal in a proper cone $\CC$. If $\L_1$ and $\L_2$ are $3$-dimensional Lorentzian cones containing $x,y,z$ and contained in~$\CC$, then $\L_1=\L_2$.
\end{lem}

\begin{proof}
Let $E$ be the affine plane generated by $x, y, z$. The sets $\L_1 \cap E$ and $\L_2 \cap E$ are solid ellipses containing $x, y, z$ in their boundary. If $H$ is a support hyperplane to $\CC$ at $x$ (resp.\ at $y, z$), then $H \cap E$ is a line tangent to both ellipses at $x$ (resp.\ at $y,z$). By the paragraph preceding the lemma, both ellipses coincide and therefore $\L_1=\L_2$.
\end{proof}

Using this lemma, we can show the next proposition:

\begin{prop} \label{proposition:three-extremals}
Let $\CC$ be a proper cone such that $(\CC_\square,\CC)$ has the LFP. If $x,y,z\in \CC$ are extremal and pairwise non-collinear, then the cone $\CC \cap \mathspan \{x,y,z\}$ is either classical or Lorentzian.
\end{prop}

\begin{example}
A special case of Theorem \ref{thm:symmetric-cones-LFP} is that $(\CC_{\square},\PSD_n(\R))$ has the LFP. We can illustrate Proposition \ref{proposition:three-extremals} in this case: Let $v_1$, $v_2$, $v_3$ be unit vectors in $\R^n$ and $\Pi_1$, $\Pi_2$, $\Pi_3$ be the corresponding rank $1$ orthogonal projections. The cone $\PSD_n(\R) \cap \mathspan\{\Pi_1,\Pi_2,\Pi_3\}$ is classical if the vectors $v_1,v_2,v_3$ are linearly independent and Lorentzian otherwise. 
\end{example}

\begin{proof}
Let $V \coloneqq \mathspan \{x,y,z\}$ and note that $\dim(V)=3$ by the assumptions. Suppose that the $3$-dimensional cone $\CC \cap V$ is not classical. Choose a vector $w$ extremal in $\CC \cap V$ and not collinear to either~$x$,~$y$ or~$z$.
By Lemma \ref{lemma:rescaling}, up to permuting and scaling the vectors $x$, $y$, $z$, we may assume that $x+y=z+w$. Since $\face(x)=\R_+x$, $\face(y)=\R_+y$, $\face(z)=\R_+z$ and $\face(w) \cap V = \R_+w$, Proposition \ref{prop:extremality-conditions} implies that the map $T = \sqmap{x}{z}{w}{y}$ is extremal in $\Pos(\CC_\square,\CC)$. Since $T$ is Lorentz-factorizable and has rank~$3$, by Lemma~\ref{lemma:lorentz-extreme}, there exists a $3$-dimensional Lorentzian cone $\L_w$ such that $\L_w \subseteq \CC \cap V$ and $x,y,z,w \in \L_w \subseteq \CC \cap V$.

By Lemma \ref{lemma:unique-ellipse}, we have $\L_{w'} = \L_w$ whenever $w'$ is extremal in $\CC \cap V$ and proportional to neither $x$, $y$, $z$. Since $\L_w$ contains every extremal element of $\CC \cap V$, we have $\CC \cap V=\L_w$, so $\CC \cap V$ is Lorentzian.
\end{proof}

\begin{proof}[Proof of Theorem \ref{theorem:strictly-convex}]
Since any cone of dimension $1$ or $2$ is Lorentzian, we may assume that \mbox{$\dim (\CC) \geq 3$}. Let $V$ be a $3$-dimensional subspace intersecting $\inter (\CC)$. Let $x,y,z$ be elements in $\partial \CC$ such that $V=\mathspan \{x,y,z\}$. Since $\CC$ is strictly convex, the points $x, y, z$ are extremal in $\CC$. By Proposition \ref{proposition:three-extremals}, the cone $\CC \cap V$ is Lorentzian or classical. The latter is excluded since $\CC$ is strictly convex. Hence, for every $3$-dimensional subspace $V$ intersecting $\inter(\CC)$, the cone $\CC \cap V$ is Lorentzian. A classical result (see for example \cite[p.\ 91]{Busemann55}) implies that~$\CC$ is Lorentzian.
\end{proof}

\subsection{The Lorentzian faces of \texorpdfstring{$\CC$}{C} when \texorpdfstring{$(\CC_\square,\CC)$}{(Csquare,C)} has the LFP}

We will now use the results from the previous sections to study the faces of proper cones $\CC$ for which $(\CC_\square,\CC)$ has the LFP. We start with an important lemma.

\begin{lem}\label{lemma:two-extreme-points}
Let $\CC$ be a proper cone such that $(\CC_\square,\CC)$ has the LFP. Then, either $\CC$ is strictly convex, or for every $x,y$ extremal in $\CC$ we have $x+y \in \partial\CC$.
\end{lem}

\begin{proof}
It suffices to show that for every $w_1,w_2,w_3$ linearly independent and extremal in $\CC$, we have
\begin{equation} \label{eq:two-extreme-points} w_1+w_2 \in \partial \CC \ \textnormal{ if and only if } \ w_1+w_3 \in \partial \CC .\end{equation}
Indeed, if $\CC$ is not strictly convex, then there exists $z_1,z_2$ linearly independent and extremal in $\CC$ such that $z_1 +z_2 \in \partial \CC$; repeated uses of \eqref{eq:two-extreme-points} imply that $x+ y \in \partial \CC$ for every $x,y$ extremal in $\CC$.

By symmetry, it suffices to show one direction in \eqref{eq:two-extreme-points}. If we assume that $w_1 + w_2 \in \int(\CC)$, then by Proposition \ref{proposition:three-extremals} the cone $\CC \cap \mathspan \{w_1,w_2,w_3\}$ is Lorentzian and in particular strictly convex. This implies that 
$\CC \cap \mathspan \{w_1,w_2,w_3\}\subseteq \face(w_1+w_3)$ and hence $w_1+w_2\in \face(w_1+w_3)$ showing that $w_1+w_3 \in \int(\CC)$.
\end{proof}

The previous lemma implies immediately that for any proper cone $\CC$ not strictly convex and such that $(\CC_\square,\CC)$ has the LFP, the set $\face(x+y)$ for $x,y$ extremal in $\CC$ is never the entire cone $\CC$. We can say a bit more in this case: 

\begin{prop}\label{prop:SomeFacesLorentz} 
Let $\CC$ be a proper cone such that $(\CC_\square,\CC)$ has the LFP. For every $x,y$ extremal in $\CC$, the face $\face(x+y)$ is Lorentzian.
\end{prop}
 
\begin{proof}
Let $F = \face(x+y)$. By Proposition~\ref{prop:LFP-passesTo}, the pair $(\CC_\square,F)$ has the LFP as well. As $x+y$ belongs to the relative interior of $F$, we conclude by Lemma \ref{lemma:two-extreme-points} that~$F$ is strictly convex, and it follows from Theorem \ref{theorem:strictly-convex} that $F$ is Lorentzian.
\end{proof}

We also have the following proposition:

\begin{prop} \label{prop:two-lorentzian-faces}
Let $\CC$ be a proper cone such that $(\CC_\square,\CC)$ has the LFP and let $F_1 \neq F_2$ be two strictly convex faces of $\CC$. Then, we have 
\[
\mathspan(F_1 \cap F_2) = \mathspan(F_1) \cap \mathspan(F_2),
\]
and 
\[ \dim (\mathspan(F_1) \cap \mathspan(F_2)) = \dim(F_1 \cap F_2) \leq 1 .\]
In particular, we have 
\[
\dim(F_1) + \dim(F_2) \leq \dim(\CC)+1 .
\]
\end{prop}

\begin{proof}

Clearly, $\mathspan(F_1 \cap F_2) \subseteq \mathspan(F_1) \cap \mathspan(F_2)$. Conversely, let $z \in \mathspan(F_1) \cap \mathspan(F_2)$. We will show that $z \in \CC$ or $-z\in \CC$, which immediately implies $z\in \mathspan(F_1 \cap F_2)$. Assume for contradiction that neither $z$ nor $-z$ belongs to $\CC$. Then, by Lemma \ref{lemma:basic-strictlyconvex}, there are elements $x_1,y_1$ extremal in $F_1$ and elements $x_2,y_2$ extremal in $F_2$ such that $z=x_1-y_1=x_2-y_2$. Note that if $x_1$ and $x_2$ were collinear, extremality would imply $x_1=x_2$, $y_1=y_2$ and thus $F_1=F_2$ since both faces are strictly convex. In a similar way, it can be shown that all vectors $x_1,x_2,y_1,y_2$ are pairwise non-collinear. The positive map $P=\sqmap{x_1}{x_2}{y_1}{y_2}$ has rank $3$ and (by Corollary~\ref{cor:extremality-conditions}) is extremal in $\Pos(\CC_\square,\CC)$. Then, Lemma \ref{lemma:ellipse-sandwiching} implies  that $x_1,y_1\in P(\CC_{\square})\subseteq \face(x_2+y_2)=F_2$ implying that $F_1=\face(x_1+y_1)\subseteq F_2$. A symmetric argument shows that $F_2\subseteq F_1$ hence $F_1=F_2$, contradicting the assumptions. 

We proved that $\mathspan(F_1 \cap F_2) = \mathspan(F_1) \cap \mathspan(F_2)$. If $\dim (F_1 \cap F_2) > 1$, then $F_1 \cap F_2$ would contain two elements $z_1,z_2$ which are linearly independent and extremal in $\CC$. Strict convexity implies that $\face(z_1+z_2)=F_1$ and $\face(z_1+z_2)=F_2$, hence $F_1=F_2$, a contradiction.
\end{proof}

Let us finish this section with a proposition. We denote by $\SCF(\CC)$ the set of strictly convex faces $F \subseteq \CC$ such that $\dim(F) \geq 3$. The following proposition shows that if some pair of elements of $\SCF(\CC)$ have a nonzero intersection, then there are many such pairs.

\begin{prop}\label{prop:Crucial}
Let $\CC$ be a proper cone such that $(\CC_\square,\CC)$ has the LFP, and let $F \neq F'$ elements of $\SCF(\CC)$ such that $F\cap F'$ contains a nonzero element $z$. 
\begin{enumerate}
\item[(i)] Consider elements $x$ extremal in $F$ and $x'$ extremal in $F'$. If $x$ and $x'$ are non-collinear, then $F_{x,x'} \coloneqq \face(x+x') \in \SCF(\CC)$.
\item[(ii)] Consider elements $x,y$ extremal in $F$ and elements $x',y'$ extremal in $F'$. If $x,x',y,y',z$ are pairwise non-collinear,  then
\[ \dim \left( F_{x,x'} \cap F_{y,y'} \right) = 1 \]
\end{enumerate}
\end{prop}

\begin{proof}
By Proposition \ref{prop:SomeFacesLorentz}, each face $F_{x,x'}$ is Lorentzian. We postpone the proof that $\dim (F_{x,x'}) \geq 3$ and first prove (ii).

Assume as in (ii) that $x,x',y,y',z$ are pairwise non-collinear. By Proposition~\ref{prop:two-lorentzian-faces}, we have $F_{x,x'} \cap F = \R_+x$ and $F_{y,y'} \cap F = \R_+y$ and therefore $F_{x,x'} \neq F_{y,y'}$. Proposition~\ref{prop:two-lorentzian-faces} also implies that $\dim (F_{x,x'} \cap F_{y,y'}) \leq 1$. Let us assume for contradiction that $F_{x,x'} \cap F_{y,y'} = \{0\}$. As $F$ and $F'$ are Lorentzian, by Lemma \ref{lemma:basic-strictlyconvex}(ii) there exist $\lambda,\lambda' > 0$, elements $w$ extremal in $F$ and $w'$ extremal in $F'$ such that
\[
x+z = w+ \lambda y \quad \text{ and }\quad y'+z = w'+\lambda' x'.
\]
Up to rescaling $x'$ and $y$ (and since $F_{x,x'}=F_{x,\lambda' x'}$) we may assume that $\lambda = \lambda'=1$. By extremality of the involved points, $w$ cannot be collinear with $x$, $y$ or $z$; and $w'$ cannot be collinear with $x'$, $y'$ or $z$. Combining these two equations, we find that 
\[
w+y+y'=w'+x+x' .
\]
We have
\begin{align*}
\face(w)\cap F_{x,x'} &=\{ 0 \},\\
\face(w')\cap F_{y,y'} &=\{ 0 \},\\
\face(w)\cap \face(w') &=\{ 0 \}, \\ 
F_{x,x'} \cap F_{y,y'} &=\{ 0 \}.
\end{align*}

Here, the first equation follows from the fact that $F_{x,x'}$ is Lorentzian (see Proposition~\ref{prop:SomeFacesLorentz}) and hence cannot intersect the face $F$ in another ray besides $\R_+ y$, and the second equation follows similarly. The third equation is clear and the fourth equation is our assumption. Lemma \ref{lemma:face_with_no_EB} implies that $y+y' \in F_{w,w'}$ and thus that $F_{y,y'} \subseteq F_{w,w'}$. Since both faces are Lorentzian, this implies $F_{y,y'}=F_{w,w'}$. Taking the intersection with $F$, we obtain that $y$ and $w$ are proportional, which is a contradiction.

The previous argument shows that $F_{x,x'} \cap F_{y,y'} \neq \{0\}$. As $x,x',y,y'$ are pairwise non-collinear, we would have $F_{x,x'} \cap F_{y,y'} = \{0\}$ if $\dim(F_{x,x'})=2$ or $\dim(F_{y,y'})=2$. Hence, we conclude that these faces have dimension at least $3$.
\end{proof}

\subsection{When \texorpdfstring{$\CC$}{C} is symmetric: Proof of Theorem \ref{thm:symmetric-cones-LFP}}

By Vinberg's classification of symmetric cones (see~\cite{vinberg1963theory}) and Proposition \ref{prop:LFP-passesTo}, it is enough to show that $(\CC_{\square},\CC)$ has the LFP when $\CC$ is Lorentzian, a cone of positive semidefinite $n\times n$ matrices over $\R,\C$ or $\quat$, or the cone $\PSD_3(\Oct)$. The first case is clear. Next, let us consider the case $\CC=\PSD_n(\C)$ for some $n\in\N$. Let $P\in \Pos(\CC_{\square},\PSD_n(\C))$ denote an extremal positive map given as 
\[
P= \Sqmap{A_1}{A_2}{A_3}{A_4},
\]
with $A_i\in \PSD_n(\C)$ for $i\in\lset 1,2,3,4\rset$.
If $\rank(P)=1$, then $P$ factors through Lorentzian cones. By Lemma~\ref{lemma:rank-not-2}, we have $\rank(P)=3$ if $\rank(P)\neq 1$. In this case, by~\cite[Example 11]{jenvcova2018incompatible}, we have $\rank(A_i)=1$ for all $i\in\lset 1,2,3,4\rset$. As $\rank(A_1+A_4)=2$, the face $\face(A_1+A_4)$ is Lorentzian as it is isomorphic to $\PSD_2(\C)\simeq \L_{4}$. Since $P(\CC_\square)\subseteq \face(A_1+A_4)$, we conclude that $P$ factors through Lorentzian cones. As all extremal maps in $\Pos(\CC_{\square},\PSD_n(\C))$ factor through Lorentzian cones, we have shown that $(\CC_\square, \PSD_n(\C))$ has the LFP. By Proposition \ref{prop:LFP-passesTo}, it follows that $(\CC_\square, \PSD_n(\R))$ has the LFP for every $n\in\N$ as $\PSD_n(\R)$ is a retract of $\PSD_n(\C)$. It is not difficult to show that $\PSD_n(\quat)$ is a retract of $\PSD_{2n}(\C)$ and by Proposition \ref{prop:LFP-passesTo} also $(\CC_\square, \PSD_n(\quat))$ has the LFP for every $n\in \N$. 

The only remaining case is $\CC=\PSD_3(\Oct)$, which is not a retract of $\PSD_n(\R)$ for any $n\in\N$. We will use the structure of faces of the cone $\PSD_3(\Oct)$, which is very similar to the structure of faces of $\PSD_3(\R)$. Most importantly, any element of $\PSD_3(\Oct)$ is equivalent to an idempotent by the action of an automorphism~\cite{baez2002octonions}. The trace of this idempotent defines the rank of elements in $\PSD_3(\Oct)$ taking values in $\lset 0,1,2,3\rset$. As shown in~\cite[Lemma 3]{labardini2012cones}, we have:
\begin{enumerate}
\item A non-zero element $A\in \PSD_3(\Oct)$ is extremal if and only if $\rank(A)=1$.
\item If $A,B\in \PSD_3(\Oct)$ have $\rank(A)=\rank(B)= 1$, then $\face(A+B)$ is isomorphic to $\PSD_2(\Oct)\simeq \L_{10}$ and hence Lorentzian. 
\item If $A,B\in \PSD_3(\Oct)$ have $\rank(A)=\rank(B)= 2$, then $\face(A)$ and $\face(B)$ are isomorphic to $\PSD_2(\Oct)$, and we have $\face(A)\cap \face(B)\neq \lset 0\rset$.
\end{enumerate}
With these properties, we can prove that $(\CC_\square,\PSD_3(\Oct))$ has the LFP. Consider an extremal positive map $P\in \Pos(\CC_{\square},\PSD_3(\Oct))$ given by
\[
P= \Sqmap{A_1}{A_2}{A_3}{A_4},
\]
with $A_i\in \PSD_3(\Oct)$ for $i\in\lset 1,2,3,4\rset$. As before, we may consider the case where $\rank(P)=3$. Consider the subspaces $E_i=\mathspan(\face(A_i))$ for $i\in\lset 1,2,3,4\rset$. By Proposition~\ref{prop:extremality-conditions}, we have $E_1\cap E_2 =\lset 0\rset$ and $E_1\cap E_3 =\lset 0\rset$. If $\rank(A_1)=2$, then we have $\rank(A_2)=\rank(A_3)=1$ and $P(\CC_\square)$ is contained in the Lorentzian face $\face(A_2+A_3)$. We reach the same conclusion if $\rank(A_4)=2$. Finally, if $\rank(A_1)=\rank(A_4)=1$, then $P(\CC_\square)$ is contained in the Lorentzian face $\face(A_1+A_4)$. These cover all cases, and we conclude that $P$ factors through Lorentzian cones. As all extremal maps in $\Pos(\CC_{\square},\PSD_3(\Oct))$ factor through Lorentzian cones, we have shown that $(\CC_\square, \PSD_3(\Oct))$ has the LFP.

\subsection{When \texorpdfstring{$\dim(\CC)\leq 5$}{dim(C) at most 5}: Proof of Theorem \ref{theorem:dimension5}}

Let $\CC$ denote a proper cone for which $(\CC_\square,\CC)$ has the LFP. By Proposition \ref{prop:Crucial}, the facial structure of $\CC$ has to be complicated if there exists a pair of Lorentzian faces with dimension greater than~$3$ and intersecting in an extremal ray of $\CC$. The example of the $6$-dimensional cone $\CC=\PSD_3(\R)$ for which $(\CC_\square,\CC)$ has the LFP by Theorem \ref{thm:symmetric-cones-LFP} with uncountably many $3$-dimensional Lorentzian faces shows that this property can be realized. We will now show that dimension $6$ is the smallest dimension for which such an example exists. This leads to an easy characterization of proper cones $\CC$ of dimension at most $5$ for which $(\CC_\square,\CC)$ has the LFP. We start with a proposition:

\begin{prop} \label{prop:at-most-one-lorentz-face}
Let $\CC$ be a proper cone such that $(\CC_\square,\CC)$ has the LFP. If $\dim(\CC) \leq 5$, then $\CC$ has at most one Lorentzian face $F$ such that $\dim(F)\geq 3$.     
\end{prop}

\begin{proof}
If $\dim (\CC) \leq 4$, the result is immediate from Proposition~\ref{prop:two-lorentzian-faces}. We now consider the case when $\dim(\CC) =5$. Let $\SCF(\CC)$ be the set of strictly convex faces of~$\CC$ of dimension $\geq 3$. Let us assume for contradiction that $|\SCF(\CC)|>1$. This condition implies that $\CC$ is not strictly convex, since the faces of a strictly convex cone are either extreme rays or the full cone. Furthermore, we have the following:
\begin{itemize}
    \item By Proposition~\ref{prop:two-lorentzian-faces}, all faces in $\SCF(\CC)$ are Lorentzian of dimension $3$.
    \item For distinct faces $F,F'\in \SCF(\CC)$, the subspaces $\mathspan(F)$ and $\mathspan(F')$ intersect. Therefore, by Proposition~\ref{prop:two-lorentzian-faces}, the intersection $F \cap F'$ is an extremal ray of~$\CC$.
    \item By Lemma~\ref{lemma:two-extreme-points}, whenever $x,y$ are extremal in $\CC$, we have $x +y \in \partial \CC$ and thus $\face(x+y) \neq \CC$.
\end{itemize}

Pick two distinct faces $F_1,F_2\in \SCF(\CC)$ and let $x_{12}\in F_1\cap F_2$ be non-zero. As $\mathspan (F_1\cup F_2)=\R^5$ we may choose non-zero elements $x_{13},x_{14}$ extremal in $F_1$ and non-zero elements $x_{23},x_{24}$ extremal in $F_2$ such that $x_{12},x_{13},x_{14},x_{23},x_{24}$ are linearly independent and span $\R^5$. Moreover, we may choose these vectors such that (see the proof of Proposition \ref{prop:Crucial}) the faces $F_3=\face(x_{13}+x_{23})$ and $F_4=\face(x_{14}+x_{24})$ belong to $\SCF(\CC)$. Finally, let $x_{34}$ be non-zero and extremal such that $F_3\cap F_4=\R_+ x_{34}$. By construction, the faces $F_1,F_2,F_3,F_4$ are pairwise distinct and the $6$ vectors $(x_{ij})_{1 \leq i < j \leq 4}$ are pairwise non-collinear. As the dimension of the ambient space is $5$, there is a linear combination $(\lambda_{ij})_{1\leq i < j \leq 4}$ such that $\sum \lambda_{ij} x_{ij} = 0$. Since the face generated by two extremal elements is not equal to $\CC$, the family $\lambda_{ij}$ must contain $3$ positive elements and $3$ negative elements. After relabeling the faces and rescaling the vectors $x_{ij}$, we have one of the following relations
\begin{align*} x_{12} + x_{13} + x_{14} &= x_{23} + x_{24} + x_{34} \\
 x_{12} + x_{23} + x_{34} &= x_{13} + x_{24} + x_{14} .
\end{align*}
The first relation leads to an immediate contradiction since all vectors in the left-hand side belong to $F_1$. Assume the second relations holds and consider the set
 \[ H = \left\{ \sqmap{a}{b}{c}{d} \st a+d=b+c, \ a \in \R_+x_{12}, \ b \in \R_+x_{13}, \ c \in F_4,\ d \in F_3 \right\} \]
 which is a face of $\Pos(\CC_\square, \CC)$. Consider an element $T=\sqmap{\mu_{12}x_{12}}{\mu_{13}x_{13}}{c}{d}$ extremal in~$H$, with $\mu_{12},\mu_{13}>0$, $c\in F_4$ and $d\in F_3$. If $\rank(T) = 3$, by Lemma \ref{lemma:ellipse-sandwiching}, the range of $T$ is contained in $F_1= \face(x_{12}+x_{13})$, and we have $c=\mu_{14}x_{14}$ and $d=\mu'_{13} x_{13}$ for some $\mu_{14},\mu'_{13} >0$. In particular, we would find the relation 
 \[
 \mu_{12}x_{12}+\mu'_{13} x_{13} = \mu_{13}x_{13} + \mu_{14}x_{14} ,
 \]
 contradicting that $x_{12}$ and $x_{14}$ are both extremal. Since $H$ does not contain an extremal rank $3$ map, it is spanned by the rank $1$ maps $\sqmap{0}{x_{13}}{0}{x_{13}}$ and $\sqmap{0}{0}{x_{34}}{x_{34}}$. Since
 \[ \Sqmap{x_{12}}{x_{13}}{x_{14}+x_{24}}{x_{23} + x_{34}}  \in H,\]
we obtain a contradiction.
\end{proof}

\begin{proof}[Proof of Theorem \ref{theorem:dimension5}]
The cases of dimensions $1$ or $2$ are trivial (every cone is symmetric) and the case of dimension $3$ is covered by Proposition \ref{proposition:three-extremals}.

Let $\CC$ be a cone with $n \coloneqq \dim(\CC) \leq 5$ such that $(\CC_\square,\CC)$ has the LFP. We consider different cases; our proof for $n=5$ uses the result for $n=4$.

(i) If $\CC$ is strictly convex, then Theorem \ref{theorem:strictly-convex} implies that $\CC$ is isomorphic to $\L_n$. In the following we assume that $\CC$ is not strictly convex. 

\smallskip

(ii) Suppose that $\CC$ has no strictly convex face of dimension $\geq 3$. Let $x$ and~$y$ be linearly independent and extremal in $\CC$, and let $F=\face(x+y)$. By Proposition~\ref{prop:SomeFacesLorentz} we have $\dim(F)=2$. We claim that~$\CC$ is isomorphic to $\R_+^n$. Otherwise, by~\cite{namioka1969tensor}, there exists an extremal map $T \in \Pos(\CC_\square,\CC)$ of rank $3$. Write $T = \sqmap{x_1}{x_2}{x_3}{x_4}$ and let $E_i \coloneqq \mathspan(\face(x_i))$. By Proposition \ref{prop:extremality-conditions}, we have
\[ \dim(E_1) + \dim(E_2) + \dim(E_3) + \dim(E_4) \leq n+1\]
Since $n+1 \leq 6$, there must exist $1 \leq i < j \leq 4$ such that $\dim(E_i)=\dim(E_j)=1$. In other words, $x_i$ and $x_j$ are extremal in $\CC$. As we observed, the face $F = \face(x_i+x_j)$ has dimension $2$. By Lemma \ref{lemma:ellipse-sandwiching}, it follows that $x_k \in F$ for every $k$, contradicting the fact that $T$ has rank $3$. 

\smallskip

In the remaining cases, by Proposition \ref{prop:at-most-one-lorentz-face}, the cone $\CC$ has a unique strictly convex face $F$, which is Lorentzian and such that $3 \leq \dim(F) \leq n-1$. 

\smallskip

(iii) In the case where $n=4$, we have $\dim(F)=3$. We claim that $\CC$ contains exactly one extremal ray not contained in $F$ and therefore that $\CC \simeq \L_3 \oplus \R_+$. Otherwise, there would exist $x,y \in \CC \setminus F$ linearly independent and extremal in $\CC$. Comparing dimensions implies that $\mathspan \{x,y\} \cap \mathspan(F) \neq 0$. By Lemma \ref{lemma:basic-strictlyconvex}, there exist $w_1,w_2$ either zero or extremal in $F$ and scalars $(\lambda_1,\lambda_2)\neq (0,0)$ such that $\lambda_1 x + \lambda_2 y = w_1-w_2$. Necessarily, the vectors $w_i$ are nonzero and the scalars $\lambda_i$ have opposite signs. Up to rescaling and relabeling, we may assume that $w_1 +x  = w_2 + y$. By Proposition \ref{prop:extremality-conditions}, the positive map $\sqmap{x}{y}{w_2}{w_1}$ is extremal in $\Pos(\CC_\square,\CC)$. Lemma \ref{lemma:ellipse-sandwiching} implies that $x,y \in \face(w_1+w_2)=F$, a contradiction.

\smallskip

(iv) The case where $n=5$ and $\dim(F)=4$ is analogous to case (iii) and here we have $\CC\simeq \L_4\oplus \R_+$.

\smallskip

(v) The remaining case is when $n=5$ and $\dim(F)=3$. We claim that $\CC$ contains exactly two extremal rays not contained in $F$ and therefore that $\CC \simeq \L_3 \oplus \R_+^2$. Otherwise, there would exist $x,y,z \in \CC \setminus F$ linearly independent and extremal in $\CC$.  Comparing dimensions implies that $\mathspan \{x,y,z\} \cap \mathspan(F) \neq \{0\}$. Using Lemma~\ref{lemma:basic-strictlyconvex}, this implies the existence of $w_1,w_2$ either zero or extremal in $F$ such that
\[
\lambda_1 x + \lambda_2 y +\lambda_3 z = w_1-w_2 ,
\]
for some scalars $(\lambda_1,\lambda_2,\lambda_3) \neq (0,0,0)$. The properties of $w_1$ and $w_2$ imply that these scalars cannot be all positive or all negative. Up to permuting and rescaling $x,y,z$, and permuting $w_1, w_2$, we may assume that 
\begin{equation} \label{eq:5vectors}
x+w_2 = w_1 + y + z .
\end{equation}
The extremality of $x$ implies that $w_2 \neq 0$. If $w_1=0$, then the positive map $P = \sqmap{x}{y}{z}{w_2}$ is extremal by Corollary \ref{cor:extremality-conditions}. Using Lemma \ref{lemma:ellipse-sandwiching}, we conclude that $z,w_2\in \face(x+y)$. As $F$ is the only Lorentzian face of dimension $\geq 3$, the face $\face(x+y)$ has to be $2$-dimensional and hence $z$ is proportional to either $x$ or $y$ contradicting the construction of these vectors.

If $w_1 \neq 0$, we may apply Lemma \ref{lemma:face_with_no_EB}. Indeed, since $\face(y+z)$ is $2$-dimensional, the spans of the faces $\face(x)$, $\face(w_2)$, $\face(w_1)$ and $\face(y+z)$ are mutually disjoint by construction of the vectors $x,y,z,w_1$ and $w_2$. Using the relation \eqref{eq:5vectors}, we may apply Lemma \ref{lemma:face_with_no_EB} to conclude that $x,y+z\in \face(w_1+w_2) = F$ contradicting again the construction of these vectors.
\end{proof}

\section{Application: Resilient pairs of proper cones} \label{section:resilence}

The \emph{maximal tensor product} of proper cones $\CC_1\subseteq V_1$ and $\CC_2\subseteq V_2$ is given by
\[
\CC_1\otimes_{\max} \CC_2 = \lset z\in V_1\otimes V_2 :(f_1\otimes f_2)(z)\geq 0 , f_1\in \CC^*_1, f_2\in \CC^*_2\rset ,
\]
and their \emph{minimal tensor product} is given by 
\[
\CC_1\otimes_{\min} \CC_2 = (\CC^*_1\otimes_{\max} \CC^*_2)^* .
\]
Both tensor products are proper cones in the algebraic tensor product $V_1\otimes V_2$. Moreover, they are associative and we may consider tensor powers $\CC^{\otimes_{\max}k}$ and $\CC^{\otimes_{\min}k}$ of a given proper cone $\CC$.

In~\cite{aubrun2023annihilating}, two of the authors studied so-called \emph{entanglement-annihilating maps}. A positive map $P\in \Pos(\CC_1,\CC_2)$ for proper cones $\CC_1$ and $\CC_2$ is called \emph{$(\CC_1,\CC_2)$-entanglement-annihilating} if for every $k\in\N$ we have
\[
P^{\otimes k}\in \Pos(\CC^{\otimes_{\max} k}_1,\CC^{\otimes_{\min} k}_2) .
\]
Examples of $(\CC_1,\CC_2)$-entanglement annihilating maps are the $(\CC_1,\CC_2)$-entanglement-breaking maps, i.e., linear maps that can be expressed as a sum of rank-$1$ positive maps $x\cdot f$ for $x\in \CC_2$ and $f\in\CC^*_1$. We denote the set of $(\CC_1,\CC_2)$-entanglement-annihilating maps by $\EA(\CC_1,\CC_2)$ and the set of $(\CC_1,\CC_2)$-entanglement-breaking maps by $\EB (\CC_1,\CC_2)$. A pair of cones $(\CC_1,\CC_2)$ is called \emph{resilient} if $\EA(\CC_1,\CC_2) = \EB (\CC_1,\CC_2)$. At the time of writing, we do not know of any non-resilient pair of cones and prior to this work all known resilient pairs of cones involved at least one classical cone or one Lorentzian cone. Finding non-resilient pairs of cones would be useful for tackling open problems in quantum information theory and we refer to~\cite{aubrun2023annihilating} for more information.

For any pair of proper cones $\CC_1$ and $\CC_2$ the cone of entanglement breaking maps $\EB(\CC_1,\CC_2)$ is the trace-dual of the cone of positive maps $\Pos(\CC_2,\CC_1)$, see \cite[Lemma 2.4]{aubrun2023annihilating}. Moreover, by~\cite[Corollary 5.5]{aubrun2023annihilating} we have  
\[
\Pos(\CC_2,\CC_1)^*=\EB (\CC_1,\CC_2)\subseteq \EA (\CC_1,\CC_2) \subseteq \LorEB(\CC_1,\CC_2) := \LorFact(\CC_2,\CC_1)^* ,
\]
Here, the cone $\LorEB(\CC_1,\CC_2)$, arising as the trace-dual of the Lorentz-factorizable maps, is called the cone of \emph{Lorentz-entanglement breaking maps}. It can be defined as all maps $P\in \Pos(\CC_1,\CC_2)$ satisfying $(\ident_V\otimes P)(\L\otimes_{\max} \CC_1)\subseteq \L\otimes_{\min} \CC_2$ for any Lorentzian cone $\L\subseteq V$. In the special case where $\CC_1=\CC$ is a symmetric cone and $\CC_2=\CC_{\square}$ we can apply Theorem~\ref{thm:symmetric-cones-LFP} to show
\[
\LorEB(\CC,\CC_\square)=\EB(\CC,\CC_\square),
\]
and, by the inclusions from above, $(\CC,\CC_\square)$ is resilient. By~\cite[Lemma~2.3]{aubrun2023annihilating}, a pair $(\CC_1,\CC_2)$ is resilient if and only if the pair $(\CC^*_2,\CC^*_1)$ is resilient. Since any symmetric cone and $\CC_\square$ are selfdual, we have shown the following proposition:  

\begin{prop}\label{prop:Resilience}
Both $(\CC_\square,\CC)$ and $(\CC, \CC_\square)$ are resilient for any symmetric cone $\CC$.
\end{prop}

To conclude this section, we will apply our results to a problem in quantum information theory. For this, we introduce the notation $\PSD_n(\C)^{\otimes_h k}$ to denote the positive semidefinite operators acting on $(\C^n)^{\otimes k}$, i.e., the tensor product of the cone of positive semidefinite matrices induced by the Hilbert space tensor product. Then, we have the following theorem:

\begin{thm}
If $P\in \Pos(\PSD_3(\C), \CC_\square)$ satisfies $P^{\otimes k}\in \Pos(\PSD_3(\C)^{\otimes_h k}, \CC^{\otimes_{\min}k}_\square)$ for all $k\in\N$, then $P\in \EB(\PSD_3(\C),\CC_\square)$.
\end{thm}

\begin{proof}
Consider a linear map $P:M_3(\C)\ra \R^3$ such that $P^{\otimes k}\lb \PSD_3(\C)^{\otimes_h k}\rb\subseteq \CC^{\otimes_{\min} k}_\square$ for all $k\in\N$. As $\LorEB(\PSD_3(\C),\CC_\square)=\EB(\PSD_3(\C),\CC_\square)$ by the previous discussion, it suffices to show that $P\in \LorEB(\PSD_3(\C),\CC_\square)$. By \cite[Corollary~4.4]{LPMH25} and since $\PSD_2(\C)\simeq \L_4$ this would follow if $PQ\in \EB\lb \PSD_2(\C),\CC_{\square}\rb$ for any $Q\in \Pos(\PSD_2(\C),\PSD_3(\C))$. By~\cite[Theorem 1.2]{woronowicz1976positive} the map $Q$ can be written as a sum of a completely positive map and a composition of a completely positive map with a transpose. Therefore, it is enough to show that $PQ\in \text{EB}\lb \PSD_2(\C),\CC_{\square}\rb$ for any completely positive map $Q:M_2(\C)\ra M_3(\C)$. Consider any positive map $W\in \Pos(\CC_{\square},\PSD_2(\C))$ and note that 
\[
(WPQ)^{\otimes k}\lb \PSD_{2}(\C)^{\otimes_h k}\rb \subseteq  W^{\otimes k}\lb P^{\otimes k}\lb \PSD_{3}(\C)^{\otimes_h k}\rb\rb\subseteq \PSD_2(\C)^{\otimes_{\min}k},
\]
where we used that for any $k\in\N$ we have $Q^{\otimes k}(\PSD_{2}(\C)^{\otimes_h k})\subseteq \PSD_{3}(\C)^{\otimes_h k}$ and $W^{\otimes k}(\CC^{\otimes_{\min} k}_{\square})\subseteq \PSD_2(\C)^{\otimes_{\min}k}$. As in the proof of~\cite[Lemma 3.2]{christandl2019composed}, we conclude that $WPQ\in \EB(\PSD_2(\C),\PSD_2(\C))$. Since this holds for all positive maps $W\in \Pos(\CC_{\square},\PSD_2(\C))$, we conclude (see~\cite[Lemma 2.4]{aubrun2023annihilating}) that $PQ\in \EB\lb \PSD_2(\C),\CC_{\square}\rb$ finishing the proof.
\end{proof}

The previous theorem has a physical interpretation following the general ideas of~\cite{plavala2025all}: A pair of $2$-outcome positive operator-valued measures (POVMs) acting on a $3$-level quantum system is jointly measurable if and only if it is impossible to obtain any non-classical correlation by using the pair of POVMs locally for each party in a multipartite Bell scenario.

\section*{Acknowledgments}

We thank Josse van Dobben de Bruyn, Anna Jen{\v{c}}ov{\'a}, and Lauritz van Luijk for interesting discussions. GA was supported in part by ANR (France) under the grant ESQuisses (ANR-20-CE47-0014-01). FLP and AMH acknowledge funding from The Research Council of Norway (project 324944). FLP also gratefully acknowledges financial support by the "Transformationsprogramm Forschung und Wissenstransfer Saar" through the Center for Quantum Technologies (QuTe). This work was supported by the Swedish Research Council under grant no. 2021-06594 while the authors were in residence at Institut Mittag-Leffler in Djursholm, Sweden during the 2026 program
on Operator Algebras and Quantum Information. 

\section*{Disclosure about the use of generative AI}

The authors used a large language model to identify inconsistencies and errors in preliminary drafts. All content, including written text and proofs, has been created and checked by the authors.

\bibliography{references}{}
\bibliographystyle{alpha}

\end{document}